\documentclass[11pt]{article}
\usepackage[utf8]{inputenc}
\usepackage[T1]{fontenc}
\usepackage{lmodern}
\usepackage{amsmath, amssymb, amsfonts, amsthm}
\usepackage[english]{babel}
\usepackage[letterpaper,top=2cm,bottom=2cm,left=3cm,right=3cm,marginparwidth=1.75cm]{geometry}
\usepackage{comment} 
\usepackage{graphicx}
\usepackage[colorlinks=true, allcolors=blue]{hyperref}
\usepackage[dvipsnames]{xcolor}
\usepackage{tikz}
\usetikzlibrary{arrows.meta}
\providecolor{navy}{RGB}{25,44,75}
\providecolor{accent}{RGB}{190,65,45}
\usepackage{microtype}

\newtheorem{theorem}{Theorem}
\newtheorem{lemma}[theorem]{Lemma}

\newtheorem{corollary}[theorem]{Corollary}

\newcommand\E{\mathbb{E}}
\newcommand\Prob{\mathbb{P}}

\newcommand\G{\mathbb{G}}

\newcommand\cF{\mathcal{F}}

\newcommand{\whp}{\text{whp}}

\newenvironment{proof2}
  {\begin{proof}[Proof of the upper bound]}
  {\end{proof}}

\makeatletter
\let\originalincludegraphics\includegraphics
\renewcommand{\includegraphics}[2][]{%
  \IfFileExists{#2.pdf}{\originalincludegraphics[#1]{#2.pdf}}{%
    \IfFileExists{#2.png}{\originalincludegraphics[#1]{#2.png}}{%
      \IfFileExists{#2.jpg}{\originalincludegraphics[#1]{#2.jpg}}{%
        \fbox{\parbox[c][4cm][c]{0.46\textwidth}{%
          \centering Missing figure file: \texttt{#2}}}%
      }%
    }%
  }%
}
\makeatother

\title{Between the Semi-Random Process and $\G(n,p)$}
\author{Margarita Akhmejanova \footnote{Matematiska institutionen, Uppsala universitet, Box 480, 751 06, Uppsala, Sweden. Email: mechmathrita@gmail.com. Funded by the Wenner-Gren Foundation.}, Andrey Kupavskii}
\date{}
\begin{document}
\maketitle

\begin{abstract}
We study a random graph process that connects the Erd\H{o}s--R\'enyi random graph with the semi-random process.  In each round, a uniformly random vertex
is revealed together with a uniformly random candidate set of $s$ possible
partners, and Builder chooses one of them. We study how this limited adaptivity $s$ changes the hitting times of monotone graph properties and how the familiar logarithmic ``coupon--collector'' delays in random graphs are reflected in this setting.
\end{abstract}

\tableofcontents
\section{Introduction}\label{Intro}
A striking feature of random graph models is the presence of \emph{phase
transitions}.  A small change in the model parameter $p(n)$ around a
\emph{threshold value} may yield a completely different behavior in the
resulting random graph $\G(n,p)$.  This phenomenon appeared in the
papers of Erd\H{o}s and R\'enyi
\cite{ErdosRenyi1959,ErdosRenyi1960}, in the emergence of the giant
component around $p(n)=1/n$.

Determining the thresholds for a variety of natural graph properties, and describing their emergence around this threshold, has been one of the major themes in random graph theory ever since its inception. A plethora of results describes when a small change in the probability $p(n)$ of an edge or link can suddenly create global connectivity, a Hamilton cycle, a perfect matching, $k$-factors or dense subgraphs such as cliques.

Recent breakthrough results by Frankston, Narayanan, Kahn and Park \cite{FKNP2021} and by Park and Pham \cite{ParkPhamKK} give a general description of the location of thresholds, up to a logarithmic factor, showing that the threshold is always within a logarithmic factor of the so-called expectation threshold.

These developments raise a natural and central question: when is the logarithmic factor truly necessary, and what underlying mechanisms dictate its presence or absence? Sometimes a full $\log n$ factor is necessary — for instance, in $\G(n,p)$, the appearance of perfect matchings or connectivity requires $p \sim (\log n)/n$ \cite{ErdosRenyi1960,BollobasRG}. In other cases, only a fractional log arises: for example, the threshold for covering $\G(n,p)$ with vertex-disjoint $K_{r+1}$ cliques (a $K_{r+1}$-factor) occurs at $p=\Theta\!\left(n^{-2/(r+1)}
  (\log n)^{2/(r(r+1))}\right)$ \cite{JohanssonKahnVu2008}. In yet other cases, no logarithmic factor is needed: for every fixed $k\ge2$, the threshold for containing the $k$-th power of a Hamilton cycle is
$p=\Theta(n^{-1/k})$ \cite{Riordan2000,KahnNarayananPark2021}. A widely held heuristic is that a (fractional) log-factor is necessary if and only if the property involves a coupon-collector–type obstruction — for instance, the need for every vertex to reach degree at least one (in the case of perfect matchings) or at least two (for Hamiltonicity). 

This naturally leads to the study of models with varying levels of adaptivity. In the \emph{semi-random graph process} suggested by Michaeli and formally introduced by Ben-Eliezer, Hefetz, Kronenberg, Parczyk, Shikhelman, Stojaković  in \cite{BenEliezer2020}, randomness is mixed with limited guidance: in each round a random vertex arrives, and the Builder chooses whom to connect it to. It has been conjectured (attributed to Prałat) that in this model, no log-factor is ever needed for spanning properties — a constant suffices — because the Builder can directly target the remaining obstructions, reducing coupon collection from $O(n\log n)$ to $O(n)$ rounds. This has been proved for some properties (such as perfect matchings), but remains open in general.

We study an \emph{interpolating} version of this process.  After the random
vertex is revealed, Builder must choose its partner from a uniformly random
candidate set of $s$ vertices, where
\[
    1\le s\le n-1.
\]
Thus for $s=1$, by Lemma \ref{offer-distribution}, the proposed edges are independent and uniform in
$\binom{[n]}2$, so $G_m^1$ is the random-edge process with
replacement. While $s=n-1$
is the ordinary semi-random process.  Varying $s$ measures how a
limited amount of choice changes hitting times.

\subsection{Model definitions and notation}

We write $\G(n,p)$ for the Erdős–Rényi random graph in which each of the $\binom{n}{2}$ possible edges is present independently with probability $p$, and $\G(n,m)$ for the uniform random graph on $[n]$ with exactly $m$ edges, chosen uniformly from all $\binom{\binom{n}{2}}{m}$ labeled graphs.  A graph property $\mathcal P$ (on a fixed vertex set) is \emph{monotone edge-increasing} if
$G\subseteq H$ and $G\in\mathcal P$ imply $H\in\mathcal P$;
\emph{monotone edge-decreasing} if
$G\subseteq H$ and $H\in\mathcal P$ imply $G\in\mathcal P$.
A function $p_c(n)$ is a threshold for a monotone graph property $\mathcal{P}$ in $G(n,p)$ if $\Prob\!\big[G(n,p(n))\in\mathcal{P}\big]\to 0$ if $p(n)/p_c(n)\to 0$, and $\Prob\!\big[G(n,p(n))\in\mathcal{P}\big]\to 1$ if $p(n)/p_c(n)\to\infty$. If the transition window length is $o\!\big(p_c(n)\big)$, we say the threshold is sharp.

\subsubsection*{Our interpolating process}

Fix an integer $s\in\{1,\ldots,n-1\}$.  The process starts from the empty
graph $G_0^s$ on $[n]$.  In round $t\ge1$:
\begin{enumerate}
  \item $u_t$ is chosen uniformly from $[n]$;
  \item conditionally on $u_t$, the candidate set $S_t$ is chosen
  uniformly from
  $
      \binom{[n]\setminus\{u_t\}}{s};
  $
  \item after seeing the complete history and the pair $(u_t,S_t)$,
  Builder chooses $v_t\in S_t$ as he wants, and the edge $\{u_t,v_t\}$ is added.
\end{enumerate}
The pairs $(u_t,S_t)$ are independent in different rounds.  Parallel
edges are allowed, but every graph property in this paper is evaluated on
the underlying simple graph.

Let $\mathcal P$ be a monotone increasing graph property on $[n]$.  For
$0< p<1$, define
\[
 \tau_p(\mathcal P,s)
 :=
 \min\left\{
 m:\sup_{\mathcal A}
 \Prob_{\mathcal A}\bigl( G_m^s\in\mathcal P\bigr)\ge p
 \right\},
\]
where the supremum ranges over all adaptive Builder strategies. We use inequalities involving $\tau(\mathcal P,s)$
as shorthand for high-probability bounds.
Specifically, $\tau(\mathcal P,s)\le m$ means that some
strategy succeeds by time $m$ with probability tending
to $1$, whereas $\tau(\mathcal P,s)>m$ means that every
strategy fails by time $m$ with probability tending
to $1$.
All asymptotic statements are as $n\to\infty$.


\subsection{Previous work}

The idea that many graph properties undergo a sudden (threshold) transition in $G(n,p)$ goes back to Erd\H{o}s--R\'enyi and Gilbert \cite{ErdosRenyi1959,Gilbert1959}.
Classical milestones include the emergence of a giant component at $p=1/n$ and connectivity at $p=(\log n)/n$ \cite{ErdosRenyi1960}.
A general existence theorem—every monotone increasing graph property has some threshold function—was proved by Bollob\'as and Thomason \cite{BollobasThomason1987}.
For subgraph containment, the threshold for a fixed $H$ is $p=n^{-1/m(H)}$ with $m(H)=\max_{F\subseteq H} e(F)/v(F)$ (see work of Ruci\'nski and the first/second moment method developed from \cite{ErdosRenyi1960,BollobasRG}).
The sharpness phenomenon was clarified via Boolean-function methods: Friedgut and Kalai showed that symmetric monotone properties have sharp thresholds  whenever $p_c(n)=n^{-o(1)}$\cite{FriedgutKalai1996}, and Friedgut's sharp-threshold criterion gave a general criterion explaining when thresholds can be coarse versus sharp \cite{Friedgut1999}.

Building on this line of work, Kahn and Kalai proposed the \emph{expectation-threshold} program: for an increasing family $\mathcal F\subset 2^{X}$, say $\mathcal F$ is $p$-small if there is a set system $\mathcal G\subset 2^{X}$ with $\mathcal F\subset \langle\mathcal G\rangle:=\{T:\exists\,S\in\mathcal G,\,S\subseteq T\}$ and $\sum_{S\in\mathcal G} p^{|S|}\le \tfrac12$; then 
$q(\mathcal F):=\sup\{p:\mathcal F\text{ is $p$-small}\}$ is the trivial lower bound  on the critical threshold $p_c(\mathcal F)$. Let also $\ell(\mathcal{F})$ be the maximum between $2$ and  the size of a largest minimal element of $\mathcal{F}$.
Kahn and Kalai conjectured a universal constant $K$ such that
\[
p_c(\mathcal F)\ \le\ K\,q(\mathcal F)\,\log \ell(\mathcal{F}),
\]
which, as emphasized in \cite{KahnKalai2006,FKNP2021}, would already (and indeed was partly motivated to) capture the Erd\H{o}s--R\'enyi thresholds for difficult properties such as perfect matchings in random $r$-uniform hypergraphs (Shamir’s problem) and bounded-degree spanning trees; these were later proved by highly problem-specific methods in \cite{JohanssonKahnVu2008,Montgomery2019}. Talagrand conjectured a fractional relaxation: $\mathcal F$ is \emph{weakly $p$-small} if there exists $g:2^{X}\to\mathbb{R}_{\ge0}$ with 
$\sum_{S\subseteq T} g(S)\ge 1$ for all $T\in\mathcal F$ and $\sum_{S\subseteq X} g(S)\,p^{|S|}\le \tfrac12$; this defines the \emph{fractional expectation-threshold} $q_f(\mathcal F):=\sup\{p:\mathcal F\text{ is weakly $p$-small}\}$, which satisfies
$
q(\mathcal F)\ \le\ q_f(\mathcal F)\ \le\ p_c(\mathcal F)
$
\cite{FKNP2021}. 
Frankston, Kahn, Narayanan and Park proved Talagrand’s fractional form of the conjecture:
\[
p_c(\mathcal F)\ \le\ K\,q_f(\mathcal F)\,\log \ell(\mathcal F),
\]
with an absolute constant $K$ \cite{FKNP2021}. 
Thus $q_f$ captures $p_c$ up to a logarithmic factor (which is tight in many natural cases), yielding a unified route to near-optimal thresholds for a range of classical properties (e.g., subgraph containment and $H$-factors, perfect hypergraph matchings, bounded-degree spanning structures, Hamiltonicity), complementing both the influence-based sharp-threshold theory and the density-driven subgraph thresholds. 
Recently, Park and Pham \cite{ParkPhamKK} proved the full Kahn–Kalai conjecture, establishing
\[
p_c(\mathcal F)\ \le\ C\,q(\mathcal F)\,\log \ell(\mathcal{F})
\]
also for the expectation threshold rather than the fractional expectation threshold. 

The \emph{semi‐random graph process} after the work \cite{BenEliezer2020} was  developed and expanded in studies of many researchers. For example, the upper bound for Hamiltonicity established in \cite{GaoKaminskiMacRuryPralat2022EJC} was improved in \cite{Frieze2022} and further strengthened in \cite{FriezeGaoMacRuryPralatSorkin2025EJC}. Perfect matchings in the semi-random process were resolved in \cite{GaoMacRuryPralat2022SIAM}, with extensions to a semirandom hypergraph model in \cite{MolloyPralatSorkin2025RSA}. The initial paper \cite{BenEliezer2020}, in particular, established that the threshold for cliques of a fixed size $d\geq 3$ has order  $n^{(d-2)/(d-1)}$ up to subpolynomial factors. The chromatic and independence numbers have also been studied, with strategies designed to make the chromatic number
large and the independence number small, see, for example, \cite{GamarnikKangPralat2024SJDM}. 
Variants clarifying the role of adaptivity include the pre- and post-positional semi-random processes \cite{GaoKoerts2024JGT}, the process without replacement \cite{GilboaHefetz2021EuroComb}, and a tree-oriented analogue, the semi-random tree process \cite{BurovaLichev2025EJC}. Sharp thresholds in adaptive random graph processes, conceptually close to our interpolation viewpoint, were obtained in \cite{MacRurySurya2023RSA}. 

Another closely related random graph model with a limited amount of choice is the edge-based Achlioptas process. In the Achlioptas process on graph $H$ with parameter $r$, in each round $r$ distinct
nonedges are sampled uniformly at random, and exactly one of them must be selected and added to the graph; we discuss its connection with our clique
results below.

\subsection{Our results}

We first consider spanning properties. For the ordinary semi-random process
$s=n-1$, linear hitting times for connectivity, fixed minimum degree,
Hamiltonicity, and the existence of a perfect matching were already known:
there is no coupon-collector obstruction in the number of rounds. Our
contribution is to determine the hitting-time order uniformly throughout the
interpolating range $1\le s\le n-1$. In the $s$-semi-random process, the
coupon-collector obstruction reappears in a specific form: it moves from the
time scale to the candidate set scale and produces the term $n(\log n)/s$.
Consequently, the transition to a linear hitting time occurs when $s$
reaches order $\log n$.

\begin{theorem}[Spanning properties]
\label{thm:main-spanning-results}
Let $k\ge1$ be fixed and $\mathrm{conn}$, $\mathrm{\delta\ge k}$, $\mathrm{Ham}$, and
$\mathrm{PM}$ denote the properties of being connected, having minimum degree
at least $k$, containing
a Hamilton cycle, and containing a perfect matching, respectively. Then uniformly for $1\le s\le n-1$,
\[
\begin{aligned}
\tau(\mathrm{conn},s)
&=\Theta\left(n+\frac{n\log n}{s}\right),
&\qquad \tau(\delta\ge k,s)
&=\Theta_k\left(n+\frac{n\log n}{s}\right),\\
\tau(\mathrm{Ham},s)
&=\Theta\left(n+\frac{n\log n}{s}\right),
&\qquad \tau(\mathrm{PM},s)
&=\Theta\left(n+\frac{n\log n}{s}\right)
\quad (n\text{ even}).
\end{aligned}
\]
Moreover, if $s=o(\log n)$, then
$
    \tau(\mathrm{conn},s)
    =(1+o(1))\frac{n\log n}{s+1}.
$
\end{theorem}


Next we consider local properties. Here we conjecture that the hitting time for a clique $K_d$, $d$ is a constant, is given by the maximum of $d-1$ explicit functions of $n$ and $s$. In the paper we prove the lower bound.

For fixed $d\ge3$,
put
\[
    r_j:=\binom d2-\binom j2
    \qquad 1\le j\le d,
\]
and define
\[
    M_d(n,s):=
    \max_{1\le j\le d-1}
    \left\{
        n^{\,2-d/r_j}s^{-r_{j+1}/r_j}
    \right\}.
\]

\begin{figure}
\centering
\begin{tikzpicture}[x=8cm,y=2.5cm]
 
  \draw[-{Latex[length=1.8mm]},navy,line width=.8pt]
    (0,0) -- (1.06,0) node[right] {$s$};
  \draw[-{Latex[length=1.8mm]},navy,line width=.8pt]
    (0,0) -- (0,1.65)
    node[above right,font=\small] {$M_{4}(n,s)$};

  \foreach \x/\lab in {
    0/1,{4/7}/n^{4/7},{8/9}/n^{8/9},1/n
  }{
    \draw[navy,line width=.6pt]
      (\x,.025) -- (\x,-.025);
    \node[below=1.3mm,font=\scriptsize]
      at (\x,0) {$\lab$};
  }

  \draw[densely dashed,gray!65,line width=.45pt]
    ({4/7},0) -- ({4/7},{6/7});
  \draw[densely dashed,gray!65,line width=.45pt]
    ({8/9},0) -- ({8/9},{2/3});

  \draw[accent,line width=1.15pt]
    (0,{4/3}) -- ({4/7},{6/7})
    -- ({8/9},{2/3}) -- (1,{2/3});

  \fill[accent] ({4/7},{6/7}) circle (.85pt);
  \fill[accent] ({8/9},{2/3}) circle (.85pt);

  \node[
    font=\scriptsize,text=accent,
    fill=white,inner sep=1pt
  ] at (.27,1.29) {$n^{4/3}s^{-5/6}$};

  \node[
    font=\scriptsize,text=accent,
    fill=white,inner sep=1pt
  ] at (.69,1.04) {$n^{6/5}s^{-3/5}$};

  \node[
    font=\scriptsize,text=accent,
    fill=white,inner sep=1pt
  ] at (.96,.84) {$n^{2/3}$};
\end{tikzpicture}

\caption{The scale $M_{4}(n,s)$ on logarithmic axes,
with switch points at $s=n^{4/7}$ and $s=n^{8/9}$.}
\label{fig:k4-scale}
\end{figure}
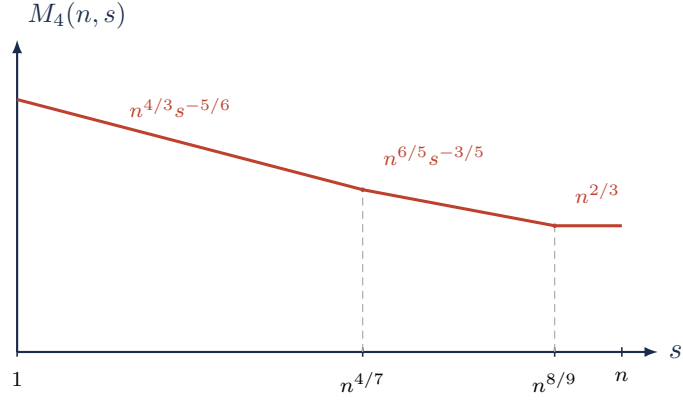

\begin{theorem}[Clique lower bound]
\label{thm:main-clique-lower}
Fix $d\ge3$.  For every function $\omega(n)\to\infty$,
\[
    \tau(K_d,s)\ge \frac{1}{\omega(n)}M_d(n,s).
\]
\end{theorem}

For $K_3$ and $K_4$, we also prove that the lower bound match the upper bound up to
arbitrarily slowly growing factor.

\begin{theorem}[Triangles and $K_4$]
\label{thm:intro-K3}
\label{thm:intro-K4}
For every $1\le s\le n-1$ and $d\in\{3,4\}$, the
high-probability hitting-time scale $\tau(K_d,s)$ is equal,
up to an arbitrarily slowly growing multiplicative factor,
to $M_{K_d}(n,s)$, where
\begin{align}
M_{K_3}(n,s)
&:=\max\left\{\;ns^{-2/3}, \sqrt{n}\right\},\\
M_{K_4}(n,s)
&:=\max\left\{
    n^{4/3}s^{-5/6},
    n^{6/5}s^{-3/5},
    n^{2/3}
\right\}.
\end{align}
\end{theorem}

Note that $M_{K_d}(n,s)=M_d(n,s)$.

Our results give the hitting-time scales $ns^{-2/3}$ for $K_3$
when $1\le s\le n^{3/4}$ and $n^{4/3}s^{-5/6}$ for $K_4$
when $1\le s\le n^{4/7}$, in both cases up to an arbitrarily
slowly growing multiplicative factor. These scales are analogous
to the corresponding results of Krivelevich and Sp\"ohel
\cite{KrivelevichSpohel2012} for the edge-based Achlioptas
process. For a subpolynomially growing number $r=n^{o(1)}$
of offered edges per round, they proved
\[
    N_0(K_3,r,n)=nr^{-2/3},
    \qquad
    N_0(K_4,r,n)=n^{4/3}r^{-5/6},
\]
where $N_0(H,r,n)$ denotes the threshold scale for constructing
a copy of $H$ in the Achlioptas process. For $d\in\{3,4\}$,
$N_0(K_d,r,n)$ plays the same role as $M_{K_d}(n,s)$ in our
model. For $K_5$ their bounds leave a gap:
\[
    n^{3/2}r^{-9/10}
    \le N_0(K_5,r,n)
    \le n^{3/2}r^{-7/8}.
\]
Note that Krivelevich and Sp\"ohel \cite{KrivelevichSpohel2012}
do not identify additional regimes for these clique thresholds
at larger values of $r$.

\section{Preliminaries}

All logarithms are natural. Floor and ceiling signs are omitted whenever they do not affect the asymptotic estimates. 
We allow Builder to play randomly. When providing upper bounds, we allow Builder to skip rounds, since the properties considered are increasing.

For $A\subseteq[n]$ and $u_t\notin A$, we repeatedly use
\begin{equation}\label{eq:hit}
 \Prob(S_t\cap A\ne\varnothing\mid u_t)
 =
 1-\frac{\binom{n-1-|A|}{s}}{\binom{n-1}{s}}
 =
 \Theta\!\left(\min\left\{1,\frac{s|A|}{n}\right\}\right).
\end{equation}

The explanation of (\ref{eq:hit}) is very simple. Put $|A|=k$. Then the upper bound follows from the union bound over the $s$ sampled vertices from $A$. The lower bound follows from the three estimates:
\begin{equation}
\label{eq:standart_exp_bound}
\frac{\binom{n-1-k}{s}}{\binom{n-1}{s}}\leq \left(1-\frac{k}{n-1}\right)^s\leq e^{-\frac{sk}{n-1}}\qquad \text{and}\qquad 1-e^{-x}\geq \frac{1}{2}\min\{1,x\}, \quad x>0.
\end{equation}

We state the next two facts as lemmas only for repeated references.
\begin{lemma}\label{offer-distribution}
For each round $t$, put
$
B_t:=\{u_t\}\cup S_t.
$
The sets $B_t$ are independent and uniformly distributed over $\binom{[n]}{s+1}$.
\end{lemma}

\begin{lemma}[Adaptive Bernoulli domination]\label{lem:adaptive-bernoulli}
Let $X_1,\ldots,X_T$ be $\{0,1\}$-valued random variables adapted to a
filtration $(\cF_t)_{t=0}^T$.  If
\[
    \Prob(X_t=1\mid\cF_{t-1})\ge p
    \qquad\text{a.s. for every }t\le T,
\]
then $\sum_{t=1}^T X_t$ stochastically dominates
$\operatorname{Bin}(T,p)$.
\end{lemma}

We also use the following standard Chernoff's inequality. The binomial
case is given in \cite[Theorem~2.1]{JLR2000}; the same
exponential-moment argument applies to sums of independent Bernoulli random variables with possibly different success probabilities, as well as to the Poisson case.

\begin{lemma}[Chernoff's inequality]\label{lem:chernoff}
Let $X$ be either a sum of independent Bernoulli random variables or Poisson random variables, and put
$\mu:=\E X$. Then, for every $t>0$,
\[
\Prob(|X-\mu|\ge t)
\le 2\exp\left(-\frac{t^2}{2(\mu+t/3)}\right).
\]
\end{lemma}

\begin{corollary}[From constant probability to whp]\label{cor:From_constant_probability_to_whp}
Let $\mathcal P$ be an increasing property.  If Builder can construct
$\mathcal P$ within $O(m)$ rounds with probability bounded away from
zero, then, for every function $\omega(n)\to\infty$, Builder can
construct $\mathcal P$ within $\omega(n)m$ rounds with high
probability.
\end{corollary}

Now suppose that in every round Builder chooses $v_t$ uniformly from $S_t$. Then for each round $t$ and for every
$x\ne y$,
\begin{equation}
\Prob\bigl(\{u_t,v_t\}=\{x,y\}\bigr)
=
\Prob(u_t=x,\ v_t=y)+\Prob(u_t=y,\ v_t=x)
=
\frac{2}{n(n-1)}
=
\frac{1}{\binom{n}{2}}.
\label{eq:random-edge}
\end{equation}
Thus the strategy produces independent
uniformly random edges of $K_n$. 

The next lemma together with the trivial
$\Omega(n)$ lower bound, gives the lower bound of order
\begin{equation}
n+\frac{n\log n}{(s+1)}
\label{eq:tau-omega-n}
\end{equation}
for hitting time for properties implying minimum degree to be at least one. In particular, it applies to connectivity, fixed minimum degree, Hamiltonicity, and the existence of a perfect matching.

\begin{lemma}[Unseen vertex]\label{lem:unseen-vertex}
Assume $s=o(n)$, and fix $\varepsilon>0$.  If $m$ equals
\[
    (1-\varepsilon)\frac{n\log n}{s+1},
\]
then whp, independently of Builder's strategy, for all $i\leq m$, there is a vertex that
has never appeared either as $u_i$ or in the offered set $S_i$, and hence
remains isolated.
\end{lemma}

\begin{proof}
Call $x\in[n]$ \emph{unseen} if $x\notin B_t$ for every $t\le m$, and
let $Z$ count the number of unseen vertices. Put $k:=s+1$. 
By Lemma~\ref{offer-distribution} and choice of $m$,
\[
    \E Z=n\left(1-\frac{k}{n}\right)^m.
\]
Since $k=s+1=o(n)$,
\[
\begin{aligned}
 \log\E Z
 &=\log n-\frac{km}{n}
   +O\left(\frac{k^2m}{n^2}\right)\\
 &\ge \varepsilon\log n-o(\log n),
\end{aligned}
\]
and hence $\E Z\to\infty$.

For distinct $x,y$, one round avoids both vertices with probability
\[
 \frac{\binom{n-2}{k}}{\binom nk}
 =\frac{(n-k)(n-k-1)}{n(n-1)}
 \le\left(1-\frac{k}{n}\right)^2.
\]
After taking the $m$-th power, the indicators that $x$ and $y$ are
unseen, are pairwise negatively correlated.  Consequently,
\[
    \operatorname{Var}Z\le\E Z.
\]
Hence, applying Chebyshev's inequality 
\[
    \Prob(Z=0)
    \le\frac{\operatorname{Var}Z}{(\E Z)^2}
    \le\frac1{\E Z}
    =o(1).
\]
\end{proof}

\section{Basic spanning properties}

In (\ref{eq:tau-omega-n}) we discussed the lower bounds. It remains to prove the corresponding upper bounds.

\subsection{Connectivity}

\begin{proof2}
 During the first $C_\delta n$ rounds, Builder chooses $v_t$ uniformly from $S_t$. By \eqref{eq:random-edge}, these rounds generate independent uniformly random edges of $K_n$. Taking $\delta=1/2$, the standard giant-component theorem \cite{JLR2000} implies that, for a sufficiently large constant $C_\delta$, the resulting graph contains a component of order at least $n/2$ whp. We take this component as the initial component $C$.

Builder then switches to the greedy attachment strategy below.
Let $C=C_t$ be the current component, put $R=R_t=[n]\setminus C$, $r=|R|$. Builder chooses an edge between $C$ and $R$ whenever $B_t=\{u_t\}\cup S_t$ intersects both $C$ and $R$. The probability $p_r$ of such a round equals

\begin{equation}
 \frac{r}{n}
 \left(1-\frac{\binom{r-1}{s}}{\binom{n-1}{s}}\right)
 +\frac{n-r}{n}
 \left(1-\frac{\binom{n-r-1}{s}}{\binom{n-1}{s}}\right),
\label{eq:pr}
\end{equation}
where $\binom{a}{s}=0$ for $a<s$.

The first bracket in
\eqref{eq:pr} is at least $1/2$, and hence the first term is at least $r/(2n)$. For the second bracket, we use (\ref{eq:standart_exp_bound}).
Since also $(n-r)/n\ge 1/2$, the second term is at least $\frac14\min\{1,sr/n\}$. Therefore
\begin{equation}
     p_r
        \ge
        \frac14\left(\frac{r}{n}+\min\left\{1,\frac{sr}{n}\right\}\right)
        \ge
        \frac14\min\left\{1,\frac{(s+1)r}{n}\right\}.
    \label{eq:pr_bound1}
\end{equation}

\smallskip
\noindent\emph{The regime $s=o(\log n)$.}
Set
$
    M:=\frac{n\log n}{s+1},
    \qquad
    r_0:=\left\lceil\frac{n}{(s+1)\log\log n}\right\rceil,
    \qquad
    q_r=\frac14\min\left\{1,\frac{(s+1)r}{n}\right\}.
$
Here $M=\omega(n)$.  From
(\ref{eq:pr_bound1}), the expected time to reduce $r$ from $n/2$ to $r_0$ is at most

\[
\begin{aligned}
\sum_{r=r_0+1}^{\lceil n/2\rceil} \frac{1}{q_r}
\le
4\sum_{r=\lceil n/(s+1)\rceil}^{\lceil n/2\rceil}1
+
\frac{4n}{(s+1)}
\sum_{r=r_0+1}^{\lceil n/(s+1)\rceil}\frac1r=
O(n)+O\left(\frac{n}{(s+1)}\log\frac{n/(s+1)}{r_0}\right)
=o(M).
\end{aligned}
\]
Therefore, by
Markov's inequality, this stage takes $o(M)$ rounds whp. 

It remains to estimate time for connecting the last $r_0$ vertices.
For this goal we need to estimate $p_r$ more accurately for $1\le r\le r_0$.
Uniformly in this range,
\[
    \frac{\binom{r-1}{s}}{\binom{n-1}{s}}
    \le\left(\frac{r}{n-1}\right)^s=o(1),
\]
so the first term in \eqref{eq:pr} is $(1+o(1))r/n$.  Conditional
on $u_t\in C$, let $X=|S_t\cap R|$. Since, $X$ is non-negative integer random variable, we have that $\mathbf 1_{\{X\ge1\}}\ge X-\binom X2.$ Taking the expectation from both sides,
\[\Prob(X\ge1)\geq  \E X- \E\binom X2.\]
Calculate
\[
 \E X=\frac{sr}{n-1},
 \qquad
 \E\binom X2
 =\binom r2\frac{s(s-1)}{(n-1)(n-2)}
 =O\left(\frac{s^2r^2}{n^2}\right).
\]
Since $sr/n\le1/\log\log n$ and $\Prob(X\ge1)\leq  \E X$, 
\[
    \Prob(X\ge1)=(1+o(1))\frac{sr}{n}
\]
uniformly for $1\le r\le r_0$.  Therefore
\begin{equation}
    p_r=(1+o(1))\frac{(s+1)r}{n},
    \qquad 1\le r\le r_0.
    \label{eq:pr_bound2}
\end{equation}

Since a crossing round decreases $R_t$ by at least one with probability $p_r$, whereas in
any other round $R_t$ remains unchanged,
\[
\E(|R_{t+1}|\mid |R_t|=r)\le r(1-p_r)+(r-1)p_r=r-p_r
\le r\left(1-(1-o(1))\frac{s+1}{n}\right).
\]
Averaging over all possible values of size of $R_t$, we obtain
\[
\E |R_{t+1}|\leq
\left(1-(1-o(1))\frac{s+1}{n}\right)\E |R_t|.
\]
For
$
    T=(1+\varepsilon/2)\frac{n\log n}{s+1},
$
where $\varepsilon>0$ is arbitrary, iterating the above estimate gives
\[
 \E |R_T|
 \le r_0\left(1-(1+o(1))\frac{(s+1)}{n}\right)^T
 =o(1).
\]
Thus $R_T=\emptyset$ whp and graph is connected after $(1+\varepsilon)M.$ By Lemma \ref{lem:unseen-vertex} and since $\varepsilon$
was arbitrary, $\tau(conn, s)=(1+o(1))\frac{n\log n}{s+1}$.

\smallskip
\noindent\emph{The regime $s=\Omega(\log n)$.}
Assume $s+1\ge c_1\log n$ for a fixed $c_1>0$.  While $r\ge n/(s+1)$,
\eqref{eq:pr_bound1} gives $p_r\ge \frac{1}{4}$.  Lemma~\ref{lem:adaptive-bernoulli}
and Chernoff's inequality \ref{lem:chernoff} show that $O(n)$ further rounds reduce $r$ to at most
$n/(s+1)$ whp.

Condition on a graph with $|R_0|\le n/(s+1)$ and restart the count.
For every $r\le n/(s+1)$, \eqref{eq:pr_bound1} gives
\[
\E(|R_{t+1}|\mid |R_t|=r)\le r-p_r
\le r\left(1-\frac{s+1}{4n}\right).
\]
Averaging over all possible values of size of $R_t$, we obtain
\[
\E |R_{t+1}|\le\left(1-\frac{s+1}{4n}\right)\E |R_t|.
\]
Therefore, for $T:=\lceil An\rceil$,
\[
\E |R_T|\le |R_0|\left(1-\frac{s+1}{4n}\right)^T
\le\frac{n}{s+1}e^{-A(s+1)/4}
\le n^{1-Ac_1/4}=o(1),
\]
provided $A>4/c_1$. By Markov's inequality, $R_T=\emptyset$ with high
probability, and hence the graph is connected after $O(n)$ rounds.

\end{proof2}

\subsection{Fixed minimum degree}\label{section:Fixed_min_degree}

\begin{proof2}
Let $G_t$ be the graph after $t$ rounds, and put
$R_t:=\{v\in[n]:d_{G_t}(v)<d\}$ and $r_t:=|R_t|$. For $u\in[n]$,
let $A_t(u):=R_t\setminus(N_{G_t}(u)\cup\{u\})$. In round $t+1$,
Builder chooses $v_{t+1}\in S_{t+1}\cap A_t(u_{t+1})$ whenever this
set is nonempty, and chooses arbitrarily otherwise.

For every fixed $G_t$,
\[
\sum_{u\in[n]}|R_t\setminus A_t(u)|
=\sum_{v\in R_t}\bigl(d_{G_t}(v)+1\bigr)\le dr_t.
\]
Hence, if $r_t>0$, at most $2d$ vertices $u$ satisfy
$|A_t(u)|<r_t/2$. Since $u_{t+1}$ is uniform on $[n]$, using
\eqref{eq:hit} we obtain
\[
\Prob\bigl(S_{t+1}\cap A_t(u_{t+1})\ne\varnothing\bigr)
\ge\left(1-\frac{2d}{n}\right)
\left(1-e^{-sr_t/(2(n-1))}\right)
\ge c\min\left\{1,\frac{sr_t}{n}\right\}.
\]

Define $\Phi_t:=\sum_{v\in[n]}(d-d_{G_t}(v))_+$. Then
$\Phi_t\le dr_t\le dn$, and
$S_{t+1}\cap A_t(u_{t+1})\ne\varnothing$ implies
$\Phi_{t+1}\le\Phi_t-1$. Consequently, for every $k$,
\begin{equation}
\Prob(\Phi_{t+1}\le k-1\mid\Phi_t=k)
\ge c_d\min\left\{1,\frac{sk}{n}\right\}.
\label{eq:min-degree-decrease}
\end{equation}

While $\Phi_t\ge n/s$, the probability in
\eqref{eq:min-degree-decrease} is at least $c_d$. Since
$\Phi_0\le dn$, at most $dn$ such decreases are required.
Lemma~\ref{lem:adaptive-bernoulli} and Chernoff's inequality \ref{lem:chernoff} therefore imply
that, for $t_0:=\lceil C_dn\rceil$, we have
$\Phi_{t_0}<n/s$ whp.

Condition on such a graph and restart the count. Since $\Phi_t$ is
non-increasing, for every $k<n/s$,
\[
\E(\Phi_{t+1}\mid\Phi_t=k)
\le k-\Prob(\Phi_{t+1}\le k-1\mid\Phi_t=k)
\le k\left(1-\frac{c_ds}{n}\right).
\]
Averaging over all possible values of $\Phi_t$, we obtain
\[
\E\Phi_{t+1}
=\sum_k\E(\Phi_{t+1}\mid\Phi_t=k)\Prob(\Phi_t=k)
\le\left(1-\frac{c_ds}{n}\right)\sum_k k\Prob(\Phi_t=k)
=\left(1-\frac{c_ds}{n}\right)\E\Phi_t.
\]
Therefore, for $T:=\lceil C_dn(\log n)/s\rceil$,
\[
\E\Phi_T
\le\Phi_0\left(1-\frac{c_ds}{n}\right)^T
\le\frac ns e^{-c_dC_d\log n}
=o(1).
\]
By Markov's inequality, $\Phi_T=0$ whp. Thus Builder
achieves minimum degree at least $d$ within
$O_d(n+n\log n/s)$ rounds.
\end{proof2}

\subsection*{A random bounded-out subgraph}

A \(k\)-out graph is a random graph  on $[n]$,
in which each vertex independently chooses $k$ distinct other vertices uniformly at random. We consider the following bipartite variant of the \(k\)-out model. Let $A\cup B=[n]$ be a fixed partition with
$\bigl||A|-|B|\bigr|\le1$.  For a fixed integer $d$, let
$\mathcal D_d(A,B)$ be the following random bipartite graph.  Every vertex
of $A$ independently chooses $d$ vertices of $B$, uniformly with
replacement, and every vertex of $B$ independently chooses $d$ vertices
of $A$, also uniformly with replacement.  We retain the underlying simple
graph of all chosen edges.

\begin{lemma}[Construction of $\mathcal D_d(A,B)$]
\label{lem:construct-d-out}
For every fixed $d$, Builder can construct a subgraph with distribution
$\mathcal D_d(A,B)$ within
\[
    O_d\left(n+\frac{n\log n}{s}\right)
\]
rounds whp.
\end{lemma}

\begin{proof}
We first generate the $d$ choices made by the vertices of $A$.  For
$a\in A$, let $\ell(a)$ be the number of choices already assigned to
$a$, and put
\[
    D_A:=\{a\in A:\ell(a)<d\}.
\]
In a round with $u_t\in B$, if $S_t\cap D_A\ne\varnothing$, choose a
target $v_t\in S_t\cap D_A$ according to a fixed rule depending only on
$S_t\cap D_A$, record $u_t$ as the next choice made by $v_t$, and
increase $\ell(v_t)$ by one.  All other rounds are ignored for the
recorded subgraph. The rest part of the proof repeats ideas and calculations which we used in Section \ref{section:Fixed_min_degree}. 

Let $r=|D_A|$ and let
\[
    \Phi_A:=\sum_{a\in A}(d-\ell(a)).
\]
Since $r\ge\Phi_A/d$, the conditional success probability in a round is
\[
\begin{aligned}
 \frac{|B|}{n}
 \left(1-\frac{\binom{n-1-r}{s}}{\binom{n-1}{s}}\right)
 &\ge c_3\min\left\{1,\frac{sr}{n}\right\}\\
 &\ge c_3\min\left\{1,\frac{s\Phi_A}{dn}\right\}
\end{aligned}
\]
for an absolute constant $c_3>0$.  The same two-phase argument as in the
Section \ref{section:Fixed_min_degree} shows that $\Phi_A$ reaches zero in
\[
    O_d\left(n+\frac{n\log n}{s}\right)
\]
rounds whp.

It remains to identify the distribution of the recorded choices.  Conditional
on $u_t=b\in B$, the distribution of $S_t\cap A$, and hence the selected
target, does not depend on the identity of $b$.  Therefore, for each $v\in A$ all recorded choices are independent and uniform on $B$.
Every target $a\in A$ occurs exactly $d$ times, so this gives precisely
the desired distribution.

Using fresh rounds, interchange $A$ and $B$ and repeat the construction.
The second family of choices is independent of the first, completing the
construction.
\end{proof}

The next lemma states that $H$ is whp
connected and has an expansion property.
The expansion estimate is obtained by adapting the
first-moment argument in the proof of
\cite[Lemma~18.17, p.~393]{FriezeKaronski}
to our bipartite model.
Connectivity is obtained by adapting the argument in the proof of
\cite[Theorem~18.2, p.~380]{FriezeKaronski}.

\begin{lemma}
\label{expanders}
Let $d=100$ and $\alpha=1/20$. If
$H\sim\mathcal D_d(A,B)$, then whp
\[
    |N_H(X)\setminus X|\ge 2|X|
    \qquad
    \text{for every }X\subseteq A\cup B,\quad
    1\le |X|\le \alpha n,
\]
and $H$ is connected.
\end{lemma}

\subsection{Hamilton cycle and perfect matching}

Bohman and Frieze \cite{Bohman2009} proved that the random
\(3\)-out graph is Hamiltonian whp.  Here we show that $\mathcal D_{100}(A,B)$ is Hamiltonian whp. From this and Lemma \ref{lem:construct-d-out} will follow the upper bound of Theorem \ref{thm:main-spanning-results} for the Hamiltonicity.

A non-edge $e$ of a graph $G$ is called a \emph{booster} if $G+e$ is
Hamiltonian or the longest path in $G+e$ is longer than the longest path
in $G$.

\begin{lemma}[P\'osa's booster lemma from \cite{Posa1976}]
\label{lem:posa-boosters}
Let $G$ be a connected non-Hamiltonian graph.  Suppose that, for some
integer $k\ge1$,
\[
    |N_G(X)\setminus X|\ge2|X|
    \qquad\text{whenever }1\le|X|\le k.
\]
Then $G$ has at least $(k+1)^2/2$ boosters.
\end{lemma}


\begin{proof2}


Fix a balanced partition $A\cup B=[n]$.  By
Lemma~\ref{lem:construct-d-out}, Builder can construct
$H\sim\mathcal D_{100}(A,B)$ in
\[
    O\left(n+\frac{n\log n}{s}\right)
\]
rounds whp.  By Lemma~\ref{expanders}, we
condition on $H$ being connected and satisfying
\[
    |N_H(X)\setminus X|\ge2|X|,
    \qquad 1\le|X|\le\alpha n,
    \qquad \alpha=\frac1{20}.
\]

Builder now ignores the choice and adds $Cn$ independent uniform random
edges, using equation~\ref{eq:random-edge}.  Expansion and connectivity are
preserved as edges are added.  If the current graph is not Hamiltonian,
Lemma~\ref{lem:posa-boosters}, with $k=\lfloor\alpha n\rfloor$, gives at
least
\[
    \frac{(k+1)^2}{2}\ge c_\alpha n^2
\]
boosters. Hence the next uniform edge is a booster with conditional probability at least a constant $p_\alpha>0$. After adding a booster, connectivity and expansion are preserved, so the argument can be repeated unless the graph is already Hamiltonian. Since each booster either creates a Hamilton cycle or increases the length of a longest path, at most $n-1$ booster steps are needed. For sufficiently large $C$, Lemma~\ref{lem:adaptive-bernoulli} and Chernoff's inequality \ref{lem:chernoff} give
\[
\Prob\bigl(\operatorname{Bin}(Cn,p_\alpha)<n\bigr)=e^{-\Omega(n)},
\]
so $O(n)$ further rounds suffice whp.
\end{proof2}


Since every Hamilton cycle on an even number of
vertices contains a perfect matching, we automatically receive the desired upper bound for perfect matchings.

\section{Local subgraph properties}

The lower bounds for $K_3$ and $K_4$ follow from Theorem \ref{thm:main-clique-lower}. For $d=3$, we have $(r_1,r_2,r_3)=(3,2,0)$; for $d=4$, we have $(r_1,r_2,r_3,r_4)=(6,5,3,0)$.

\subsection{Triangles}
The lower bound follows from the general Theorem \ref{thm:main-clique-lower}. Here we prove the matching (up to any growing function $\alpha$) upper bound. Start with the case $s\le n^{3/4}$.  Put
\[
    m:=ns^{-2/3},
    \qquad
    a:=\gamma n/s,
\]
where $\gamma>0$ is a sufficiently small constant, and fix a
set $R$ of $a$ star centers. Choose $0<\varepsilon\ll\gamma$.  Call a round useful if
$u_t\notin R$, the vertex $u_t$ has not previously been used as a
leaf, and $S_t\cap R\ne\varnothing$.  In every useful round, join
$u_t$ to a uniformly chosen vertex of $S_t\cap R$, and regard
$u_t$ as a new leaf.  Continue until $N:=\varepsilon m$ useful
rounds have occurred.

At all times at most $a+N\le2\gamma n$ vertices are unavailable as
new leaves, while \eqref{eq:hit} gives
\[
    \Prob(S_t\cap R\ne\varnothing\mid
    u_t\text{ is available})=\Theta(1).
\]
Thus every round is useful with conditional probability bounded away
from zero.  Consequently, using adapted Bernoulli
variables from Lemma \ref{lem:adaptive-bernoulli} and Chernoff's inequality \ref{lem:chernoff} show that $N$ useful rounds, and hence $N$ distinct
leaves, are obtained in $O(m)$ rounds $\whp$.

Conditional on a round being useful, its center is uniform on $R$,
and the successive center choices are independent.  Hence the star
sizes $(D_r)_{r\in R}$ have the occupancy distribution of $N$
independent balls in $a$ bins.

Let
\[
    Z:=\sum_{r\in R}\binom{D_r}{2}.
\]
Writing $Z$ as the sum of the indicators that two balls occupy the
same bin gives
\[
 \E Z=\frac{\binom N2}{a}
      =\Theta\!\left(\frac{m^2s}{n}\right),
 \qquad
 \operatorname{Var}Z\le\E Z.
\]
Since $m^2s/n=n/s^{1/3}\to\infty$, Chebyshev's inequality gives
\begin{equation}\label{eq:K3-Z}
    Z=\Theta(m^2s/n)\qquad\whp.
\end{equation}
Moreover, the mean load is
\[
    N/a=\frac{\varepsilon}{\gamma}\frac{ms}{n}
       =\frac{\varepsilon}{\gamma}s^{1/3}.
\]
As $s^{1/3}\le n/s$ for $s\le n^{3/4}$, choosing
$\varepsilon/\gamma$ small and applying the binomial upper-tail bound
and a union bound gives
\begin{equation}\label{eq:K3-load}
    \max_{r\in R}D_r\le\eta n/s\qquad\whp,
\end{equation}
where $\eta$ is an arbitrarily small fixed constant.

Keep the stars fixed.  A further round closes a triangle whenever
$u_t$ and some vertex of $S_t$ are distinct leaves of the same
star.  On \eqref{eq:K3-Z}--\eqref{eq:K3-load}, its conditional success
probability is, by \eqref{eq:hit},
\begin{align*}
 q_3
 &=
 \frac1n\sum_{r\in R}D_r
 \left(1-\frac{\binom{n-D_r}{s}}{\binom{n-1}{s}}\right)\\
 &\ge c\frac{s}{n^2}\sum_{r\in R}D_r(D_r-1)
  =\Omega\!\left(\frac{m^2s^2}{n^3}\right).
\end{align*}
The choice of $m$ is exactly the balance
\[
    mq_3=\Omega\!\left(\frac{m^3s^2}{n^3}\right)=\Omega(1).
\]
Thus another $O(m)$ rounds close a triangle with probability bounded
away from zero.  Hence,  Corollary~\ref{cor:From_constant_probability_to_whp} gives
$\tau(K_3,s)\le\alpha ns^{-2/3}$.

For $s\ge n^{3/4}$, thin the candidate set to
$s_0=\lfloor n^{3/4}\rfloor$.  It gives
$\omega(n)\sqrt n$ rounds.

\subsection{\texorpdfstring{$K_4$}{K4}}

We will need two lemmas. The first lemma allows Builder to construct an independent
random subgraph $F_p$ of a prescribed graph $F$. Lemma \ref{lem:bernoulli_subgraph_coupling} is then used directly for the random page edges
and for the final closing edges. We just will check that the restriction the maximum out-degree holds. The second lemma is an application of the first lemma for building many good stars and then pages inside them.

For a graph $F$ and $p\in[0,1]$, let $F_p$ denote the random
spanning subgraph of $F$ obtained by retaining each edge
independently with probability $p$.

\begin{lemma}[Bernoulli subgraph coupling]\label{lem:bernoulli_subgraph_coupling}
There are absolute constants $\eta,\kappa>0$ with the following
property.  Let $F$ be a graph on $[n]$ with an orientation $\vec F$ satisfying
\[
    \Delta^+(\vec F)\le \eta\frac ns,
\]
and suppose that $m\to\infty$ and $ms/n^2\le\eta$.  There is a Builder strategy for which the graph $G_{2m}$ obtained
after $2m$ rounds can be coupled with $F_p$ so that
\[
    \Prob(F_p\subseteq G_{2m})\ge 1-e^{-\Omega(m)}
    \qquad and \qquad p=\kappa\frac{ms}{n^2}.
\]
\end{lemma}

\begin{proof}
First define the rule for one round. When $u_t = x$, Builder chooses uniformly one vertex from the intersection between neighbors of $x$ and the current candidate set, i.e. from $S_t\cap N_{F}^+(x)$. If the intersection is empty, do nothing.
Let
$xy$ be a fixed edge of $F$. For $xy$ to be chosen in a given round,
three events must occur: $u_t = x$, $y\in S_t$ and among all neighbors of $x$, $y$ is taken. Let us assume that the first two events have occurred. Let $Z$ be the number of other vertices $S_t\cap N_{F}^+(x)\setminus\{y\}$. Then the intersection contains $1 + Z$ valid
vertices, and the conditional probability of choosing $y$ is $\frac{1}{(1 + Z)}.$ Therefore, using the Law of Total Probability

\begin{gather*}
q_{xy} := \Prob(\text{xy chosen in one round}) =\\\sum_z\Prob(\text{y chosen}| Z=z, u_t=x, y\in S_t)\Prob(Z=z|u_t=x, y\in S_t) \Prob(u_t=x, y\in S_t)=\\
\sum_z\frac{1}{1+z}\Prob(Z=z|u_t=x, y\in S_t)\frac{s}{n(n-1)}=\frac{s}{n(n-1)}\E\frac{1}{1+Z}.
\end{gather*}
Now let us estimate $\E Z$. 
\[\E Z =\frac{(s-1)(d^+_F(x)-1)}{n-2} \leq 2\eta.\]
Since $(1 + z)^{-1}$ is convex, Jensen’s inequality gives
\[\E \frac{1}{1 + Z}\geq\frac{1}{1 + 2\eta}.\]
Consequently,
\[
\frac{1}{1 + 2\eta}\frac{s}{n(n-1)}\leq q_{xy}
\leq\frac{s}{n(n-1)}.
\]
Independently sample $N\sim Poisson(m).$ During the first $N$ rounds apply the above rule. If $N>2m$ declare the coupling unsuccessful. By Poisson tail bound $\Prob(N > 2m) \leq e^{-cm}$. 

For each $x\to y\in E(\vec F)$, let $C_{xy}$ be the number of times this edge is selected during the first $N$ rounds. By Poisson splitting, these counts are independent, with
\[
C_{xy}\sim\operatorname{Poisson}(mq_{xy}).
\]
Thus the indicators 
\[I_{xy}:=I_{C_{xy}>0}\] 
are independent Bernoulli
variables with parameter 
\[r_{xy}=1-e^{-mq_{xy}}.\] 

By the hypothesis of Lemma \ref{lem:bernoulli_subgraph_coupling}, $\frac{ms}{n^2} \leq\eta $. Hence, after making $\eta$ small enough, there is an absolute $\kappa>0$ such
that
$r_{xy} \geq \kappa
\frac{ms}{n^2}$
for every arc. 

Put 
\[
p:=\kappa
\frac{ms}{n^2}.
\]
Independently retain each positive indicator: if $I_{xy}=1$,  keep the edge $xy$ with probability $p/r_{xy}$. It gives the desired $F_p$.
\end{proof}

\medskip
\begin{lemma}[Random or disjoint stars]\label{lem:random_stars}
Let $R,W\subseteq[n]$ be disjoint, with $|W|=\Theta(n)$ and
$|R|\le\eta n/s$.  Suppose that an application of Lemma \ref{lem:bernoulli_subgraph_coupling} to
the complete bipartite graph between $R$ and $W$ gives expected
degree at least $8k$ at every $r\in R$, where $k\gg\log n$. Let $B$ denote the resulting bipartite graph.
Then, with probability $1-o(1)$, Builder obtains sets
\[
 D_r\subseteq W,\qquad |D_r|=k\quad(r\in R),
\]
with every edge from $r$ to $D_r$ present.  Moreover:
\begin{itemize}
 \item conditioned on $d_B(r)\geq k$ for all $r\in R$, sets $D_r$ may
be chosen independently and uniformly among the $k$-subsets of $W$, or
 \item sets $D_r$ being pairwise disjoint, provided $|R|k=o(n)$.
\end{itemize}
\end{lemma}

\begin{proof}
Orient the complete bipartite graph $K_{W,R}$ from $W$ to $R$. Then,  $\Delta^+(\vec K_{W,R})=|R|\le\eta n/s$. Lemma~\ref{lem:bernoulli_subgraph_coupling} gives an independent bipartite subgraph $B$ in which every edge has
probability $p$. For each $r$,
$d_B(r)\sim Bin(|W|,p)$, and tail bound gives
\[
 \Prob(d_B(r)<k)\le e^{-ck}.
\]
A union bound shows that all centers have at least $k$ neighbours.

For the first case, choose $D_r$ uniformly among the
$k$-subsets of $N_B(r)$. The edge sets incident with distinct
centers are independent. Conditional on $d_B(r)\ge k$, every member of $\binom Wk$ has the same
probability. Thus the $D_r$'s are conditionally
independent uniform $k$-subsets.

For the second case, order the centers and choose each time from unused neighbors.
Before treating a center $i\in |R|$, at most $|R|k=o(n)$ vertices have been
used. The used set depends only on edges incident with earlier centers
and is therefore independent of the edges incident with $i$. So, the
number of unused neighbors of $i$ is binomial with mean
\[
 (|W|-o(n))p=(1-o(1))|W|p\ge7k.
\]
By Chernoff's inequality \ref{lem:chernoff}, its probability of being smaller than $k$ is $e^{-\Omega(k)}=o(n^{-1})$. A
union bound completes the proof. 

The failure in both cases
has probability $o(1)$.
\end{proof}

Here we start the proof of the upper bound of Theorem \ref{thm:intro-K4}. 

For $s\le n^{8/9}$, put
\[
 m=\begin{cases}
 n^{4/3}s^{-5/6},&1\le s\le n^{4/7},\\
 n^{6/5}s^{-3/5},&n^{4/7}\le s\le n^{8/9}.
 \end{cases}
\]
By Corollary~\ref{cor:From_constant_probability_to_whp}, it suffices to construct a $K_4$ with probability
bounded away from zero in $O(m)$ rounds. The range $s>n^{8/9}$ will
follow by reducing the number of candidates.

Let $\eta,\kappa$ be the constants in Lemma~\ref{lem:bernoulli_subgraph_coupling}, and choose sufficiently
small constants $\varepsilon,\delta>0$ in terms of them. Take disjoint
sets $R,X,Y$ with
\[
 |R|=a:=\left\lfloor\varepsilon\frac ns\right\rfloor,
 \qquad |X|=|Y|=\lfloor n/3\rfloor,
 \qquad p=\kappa\frac{ms}{n^2}.
\]
We first build stars centred at $r\in R$, with leaf sets
$A_r\subseteq X$ and $B_r\subseteq Y$. If $x\in A_r$ is subsequently
joined to two distinct vertices $y,z\in B_r$, then $r,x,y,z$ span a
$K_4$ with only $yz$ possibly missing. Let $H$ be the graph of the
pairs $yz$ obtained in this way. Our aim is to arrange that
\begin{equation}\label{eq:closing}
 \Delta(H)\le\eta n/s,\qquad p|E(H)|\ge c>0.
\end{equation}
Lemma~\ref{lem:bernoulli_subgraph_coupling} then finds an edge of $H$ in at most $2m$ further rounds
with probability at least $1-e^{-c}-o(1)$, completing a $K_4$.
All stages use separate intervals of rounds. We work with the random
subgraphs supplied by Lemmas~\ref{lem:bernoulli_subgraph_coupling} and~\ref{lem:random_stars}; the total probability that
one of these finitely many couplings fails is $o(1)$.

\paragraph{The case $1\le s\le n^{4/7}$.}
Set $d=\lfloor\delta ms/n\rfloor$. Thus
\[
 d=\Theta(n^{1/3}s^{1/6})\gg\log n,\qquad
 d\le\delta n/s,\qquad ad=\Theta(m),\qquad
 dp=\Theta((s/n)^{1/3})=o(1).
\]
Apply Lemma~\ref{lem:random_stars} between $R$ and each of $X,Y$, choosing $A_r,B_r$
as independent uniform $d$-sets. Its expected degree requirement holds
because $p|X|=p|Y|\ge8d$ for sufficiently small $\delta$.

We would like to sample the edges between $A_r$ and $B_r$ independently for each $r$. To avoid using the same edge twice, discard $xy$ whenever
$x\in A_r$ and $y\in B_r$ for more than one $r$.
For a fixed $r,x,y$, the probability of such a repetition is
$O(ad^2/n^2)=o(1)$. Let $W$ count the discarded occurrences $(r,x,y)$. Then
$\mathbb E W=o(ad^2)$, so Markov's inequality gives $W=o(ad^2)$ whp. Each pair $(r,x)$ left with fewer than $d/2$ possible neighbors in $B_r$ loses at least $d/2$ occurrences. Thus there are at most $2W/d=o(ad)$ such pairs. Hence, whp, all but $o(ad)$ pairs $(r,x)$, $x\in A_r$, still have at least $d/2$ possible neighbors in $B_r$. Discard the other pairs, and let $F$ be the union of the remaining edges, oriented from $Y$ to $X$. 

For each $y\in Y$, the number of sets $B_r$ containing $y$ has
distribution $\operatorname{Bin}(a,d/|Y|)$, and each contributes at most $d$ to
$d_F^+(y)$. The standard binomial tail bound therefore gives
\[
 \Prob\bigl(d_F^+(y)>\eta n/s\bigr)
 \le\left(C\frac{d^2}{n}\right)^{\eta n/(sd)}=o(n^{-2}),
\]
since $d^2/n=O(n^{-1/7})$ and $\eta n/(sd)\ge\eta/\delta$.
Thus $\Delta^+(F)\le\eta n/s$ with high probability, and we can
apply Lemma~\ref{lem:bernoulli_subgraph_coupling} to obtain $F_p$.

For each retained pair $(r,x)$, if exactly two of its remaining
edges to $B_r$ are sampled, say $xy$ and $xz$, put $yz$ in $H$.
Each pair has between $d/2$ and
$d$ possible neighbours, so the probability of obtaining exactly two
is $\Theta(d^2p^2)$. These events are independent, since different
pairs use disjoint sets of possible edges. Let $Z$ be the number
of pairs with exactly two sampled edges. Since
$ad^3p^2=\Theta(n^{2/3}s^{-1/6})\ge c_1n^{4/7}$, Chernoff's
inequality gives $Z=\Theta(ad^3p^2)$ with high probability.

In fact, these pairs give distinct edges of $H$ with high probability.
Every repetition requires four distinct sampled edges. Counting
separately pairs with the same centre $r$ and with different centres,
the expected number of repetitions is at most
\[
 O\left(ad^4p^4+\frac{a^2d^6p^4}{n^2}\right)
 =O\left((s/n)^{1/3}+n^{-2/3}s^{-1/3}\right)=o(1).
\]
For the second term we used that two independent $d$-subsets of $Y$
contain $O(d^4/n^2)$ common unordered pairs on average.
Consequently $|E(H)|=Z=\Theta(ad^3p^2)$ with high probability.
Also $d_H(y)\le d_{F_p}^+(y)\le d_F^+(y)$: each pair $(r,x)$
contributes at most one edge to $H$, using its own sampled edges.
Finally,
\[
 p|E(H)|=\Theta(ad^3p^3)
 =\Theta\left(\frac{m^6s^5}{n^8}\right)=\Theta(1).
\]
This proves~\eqref{eq:closing} and completes the construction in this range.

\paragraph{The case $n^{4/7}\le s\le n^{8/9}$.}
In this range we take $|A_r|=\Theta(n/s)$ and larger sets $B_r$.
Recall that round $t$ offers a vertex $u_t$ and a uniformly random
set $S_t\subseteq[n]\setminus\{u_t\}$ of $s$ possible neighbors,
from which Builder chooses a neighbor of $u_t$.
When $u_t\in B_r$, the set $S_t$ meets $A_r$ with probability
bounded away from zero. A larger $B_r$ increases the probability
that $u_t\in B_r$. We use vertex-disjoint stars and add edges from distinct
vertices of $B_r$ to independently and uniformly chosen vertices
of $A_r$. If $y,z\in B_r$ are joined to the same $x\in A_r$, then
adding $yz$ completes a $K_4$ on $r,x,y,z$. These pairs form
disjoint cliques, one on each group of leaves
with a common chosen neighbour. We want many pairs in these groups,
while keeping every group small enough to apply Lemma~\ref{lem:bernoulli_subgraph_coupling}.

Set
\[
 b=\left\lfloor\delta\frac ns\right\rfloor,\qquad
 d=\left\lfloor\delta\frac{ms}{n}\right\rfloor.
\]
Here $b\le d$, both are much larger than $\log n$, and
$ab\le ad=\Theta(m)=o(n)$. Lemma~\ref{lem:random_stars} therefore gives pairwise
vertex-disjoint stars with centre $r$ and leaf sets $A_r\subseteq X$,
$B_r\subseteq Y$, where $|A_r|=b$ and $|B_r|=d$.
As before, the expected sampled degrees are at least $8d$.

For each $r$, we will join exactly $h$ distinct vertices of $B_r$
to vertices of $A_r$, where
\[
 h=\left\lfloor\gamma\frac{m^2s}{n^2}\right\rfloor
 \qquad\text{and}\qquad h/d=\Theta(m/n)=o(1).
\]
The constant $\gamma>0$ will be chosen sufficiently small in terms
of $\delta$ and $\eta$. Here $h=\Theta(md/n)$, of the same order as
the expected number of offers from $B_r$ in $m$ rounds.

Whenever $u_t=y\in B_r$ has not yet been
joined to $A_r$ by this procedure and $S_t$ meets $A_r$, add $yx$,
choosing $x$ uniformly from $S_t\cap A_r$. Ignore subsequent offers
of $y$, and stop at $r$ once $h$ vertices of $B_r$ have been joined to $A_r$.
The sets $B_r$
are disjoint, so there is at most one applicable centre in a round.
In all other rounds play arbitrarily and ignore the resulting edges.

This procedure joins $h$ distinct vertices of each $B_r$ to $A_r$
in $Cm$ rounds with high probability, for a sufficiently large constant $C$.
Indeed, until $h$ vertices of $B_r$ have been joined to $A_r$, at least $d-h\ge d/2$
vertices of $B_r$ remain available. The offered vertex is one of these with
probability at least $d/(2n)$, and $S_t$ meets $A_r$
with probability bounded away from zero, since $bs/n=\Theta(1)$.
Hence the chance of adding an edge at this centre in a round is at
least $c_2d/n$. Lemma~\ref{lem:adaptive-bernoulli} and Chernoff's inequality bound the probability
of adding fewer than $h$ edges after $Cm$ rounds by $e^{-\Omega(h)}$.
A union bound suffices, since
$h=\Theta(n^{2/5}s^{-1/5})\ge c_3n^{2/9}$.

The chosen neighbours in $A_r$ are independent and uniform.
To see this, condition on which vertices of the sets $B_r$ are joined
to their respective sets $A_r$, and in which rounds. This information
depends only on the offered vertices and on whether $S_t$
meets $A_r$, not on the neighbour chosen from it.
Symmetry within $A_r$ gives a uniform choice in each of these rounds,
independently across rounds and centres.
For $x\in A_r$, let $Q_{r,x}\subseteq B_r$ be the set of vertices
joined to $x$ by this procedure, and let
\[
 H=\bigcup_{r\in R}\ \bigcup_{x\in A_r}K_{Q_{r,x}}.
\]
As observed above, each edge of $H$ completes a $K_4$. Moreover,
the sets $Q_{r,x}$ are disjoint, so $|E(H)|$ counts pairs with a
common chosen neighbour, and $\Delta(H)=\max_{r,x}(|Q_{r,x}|-1)$.

We now check the two bounds in~\eqref{eq:closing}, conditional on
having joined $h$ vertices of each $B_r$ to $A_r$. Any two of these
vertices of $B_r$ have the same chosen neighbour with probability $1/b$.
The indicators
of this event for different pairs are pairwise independent, even
when the pairs share a vertex: the joint probability is then $1/b^2$.
Thus, for $Z=|E(H)|$,
\[
 \E Z=a\binom h2\frac1b
 =\Theta\left(\frac{m^4s^2}{n^4}\right),\qquad
 \operatorname{Var} Z\le\E Z.
\]
Since $m^4s^2/n^4\ge n^{4/9}$, Chebyshev's inequality gives
$|E(H)|=\Theta(m^4s^2/n^4)$ with high probability.
It remains to bound the size of each group. The random variable
$|Q_{r,x}|$ has distribution $\operatorname{Bin}(h,1/b)$, with mean $h/b$.
Compared with the allowed group size $\eta n/s$, this mean is
small, since
\[
 \frac{h/b}{n/s}
 =\left(\frac\gamma\delta+o(1)\right)n^{-8/5}s^{9/5}
 \le\frac{2\gamma}{\delta}.
\]
Choosing $\gamma$ sufficiently small, a binomial tail bound and a
union bound over the $ab\le n$ sets give
$\max_{r,x}|Q_{r,x}|\le\eta n/s$ with high probability; here
$n/s\ge n^{1/9}$. Hence $\Delta(H)\le\eta n/s$, while
\[
 p|E(H)|=\Theta\left(\frac{m^5s^3}{n^6}\right)=\Theta(1).
\]
Again~\eqref{eq:closing} completes the construction.

\paragraph{The case $s>n^{8/9}$.}
In every round keep a uniformly chosen subset of
$s_0=\lfloor n^{8/9}\rfloor$ candidates. This simulates the process
with parameter $s_0$, so the preceding construction takes
$O(n^{2/3})$ rounds and succeeds with probability bounded away from
zero. Corollary~\ref{cor:From_constant_probability_to_whp} completes the proof in all three ranges.
\hfill$\square$

\subsection{The lower bound for cliques}

To prove Theorem~\ref{thm:main-clique-lower}, we develop a way to compress the information about a clique
produced by the process. The output is a certificate in which some vertex names are replaced by free variables, some edges are omitted, the remaining edges are listed chronologically with their orientations.

Let us illustrate it on an example. Suppose that a copy of $K_4$ on vertices $a,b,c,d$
is formed by the moves
\[
 a\to c,\qquad c\to d,\qquad b\to d,\qquad
 a\to d,\qquad b\to c,\qquad a\to b,
\]
in this order. We keep the names $a,b$, replace $c,d$ by free
variables $x,y$, and omit the edge $cd$. The {\it certificate with two free variables} records
\[
 a\to x,\qquad b\to y,\qquad
 a\to y,\qquad b\to x,\qquad a\to b,
\]
together with the rounds in which these five edges were added. 

We first prove a lemma and then use it to construct the certificates.
Here $j$ denotes the number of free variables, while the remaining
$d-j$ vertex names are specified. Varying $j$ over $1,\ldots,d-1$
gives a family of probability bounds, from which we choose the strongest.
For an oriented edge $x\to y$, call $x$ its tail and $y$ its head.

\begin{lemma}\label{lem:clique-order}
Let the edges of $K_d$ be oriented and linearly ordered, and let $1\le j\le d-1$. There exist a set $B$ of $d-j$ vertices and at least $r_j=\binom d2-\binom j2$ retained edges such that every vertex outside $B$ occurs among the retained edges and its first occurrence, in their inherited order, is as the head of an edge.

Consequently, the tail of every retained edge belongs to $B$ or occurs in an earlier retained edge. Exactly $j$ retained edges contain the first occurrence of a vertex outside $B$. In each of the remaining edges, of which there are at least $r_j-j=r_{j+1}=\binom d2-\binom{j+1}{2}$, each endpoint belongs to $B$ or occurs in an earlier retained edge.
\end{lemma}

\begin{proof}
We choose the free variables and retained edges by working
backwards. Let $\beta(H)$ be the minimum size of a vertex cover
of $H$, and start with $H=K_d$ and $\ell=j$. Take the last remaining
edge $x\to y$. If $\beta(H-xy)\ge\ell$, delete only this edge.
Otherwise delete $y$ and all its incident edges, designate its
name as a free variable, and decrease $\ell$ by one. In either
case, retain $x\to y$. Repeat until no edges remain, keeping
isolated vertices throughout.

We maintain $\beta(H)\ge\ell\ge0$. At a vertex deletion,
$\beta(H)=\ell$ and $\beta(H-xy)=\ell-1$. A minimum cover of
$H-xy$ avoids $x,y$ and contains every neighbour of $y$ except $x$.
Thus $d_H(y)\le\ell$ and $\beta(H-y)=\ell-1$, preserving the
inequality.

At termination $\ell=0$, so exactly $j$ vertices have been deleted.
These deletions remove at most
$j+(j-1)+\cdots+1=\binom{j+1}{2}$ edges, leaving at least
$r_{j+1}$ individual edge deletions. Specify the names of the
surviving vertices and read the retained edges in their original
order. Each deleted vertex is first encountered as the head of
the edge that caused its deletion, whose tail is already identified.
Every individually deleted edge has both endpoints already
identified. There are therefore at least $j+r_{j+1}=r_j$ retained
edges.
\end{proof}

We construct a certificate using the backward procedure from the proof
of Lemma~\ref{lem:clique-order}. It consists of the surviving vertex names
and the retained oriented edges with their birth rounds listed
chronologically. Throughout the list, each deleted vertex name is
replaced by a distinct free variable.

Come back to our $K_4$ example. In the $K_4$ example, at the first recorded round we require
that $a$ is offered and assign Builder's chosen neighbour to $x$. At the second, we require that $b$ is offered and assign Builder's chosen neighbour to $y$. These choices must give four distinct
vertices $a,b,x,y$. Each of these two checks has conditional probability at most $1/n$. The remaining three recorded moves must be $a\to y$, $b\to x$, and $a\to b$, respectively, with these same values of $x,y$. Both endpoints are now determined, so
each check has conditional probability at most $s/[n(n-1)]$.
Thus a fixed certificate is satisfied with probability at most
\[
 \frac{1}{n^2}\left(\frac{s}{n(n-1)}\right)^3.
\]

\begin{proof}[Proof of Theorem~\ref{thm:main-clique-lower}]
Fix a Builder strategy and $1\le j\le d-1$ and suppose that
$ms\le n(n-1)$. Use $[d]$ as the abstract vertex set.
There are $O_d(1)$ choices of the set $B$ of $d-j$ vertices
and the ordered, oriented retained edges, and at most $n^{d-j}$
injective maps $\varphi:B\to[n]$.

Fix these choices and the recorded rounds. Read the retained
edges chronologically, extending $\varphi$ according to
Builder's choices. For each retained edge $x\to y$, consider its recorded round. By Lemma~\ref{lem:clique-order}, $\varphi(x)$ is already defined,
and we require that the offered vertex is $\varphi(x)$.
Let $v$ be Builder's chosen vertex. If $\varphi(y)$ is already
defined, we require $v=\varphi(y)$. If $\varphi(y)$ is not yet
defined, we require that $v$ is outside the current image of
$\varphi$ and extend $\varphi$ by setting $\varphi(y)=v$.
The certificate is satisfied if all these checks succeed.

Conditional on the full preceding history, each of the $j$
checks that extend $\varphi$ succeeds with probability at most
$1/n$, and each other check with probability at most
$s/[n(n-1)]$. For a certificate with $q$ retained edges, summing
over at most $m^q$ choices of the recorded rounds therefore gives
\[
 \left(\frac mn\right)^j
 \left(\frac{ms}{n(n-1)}\right)^{q-j}
 \le
 \left(\frac mn\right)^j
 \left(\frac{ms}{n(n-1)}\right)^{r_{j+1}},
\]
since $q-j\ge r_{j+1}$ and $ms/[n(n-1)]\le1$.
Every clique yields a satisfied certificate by
Lemma~\ref{lem:clique-order}, so a union bound gives
\begin{align}
 \Prob(K_d\subseteq G_m^s)
 &\le C_d n^{d-j}\left(\frac mn\right)^j
       \left(\frac{ms}{n(n-1)}\right)^{r_{j+1}}\notag\\
 &\le C'_d\frac{m^{r_j}s^{r_{j+1}}}{n^{2r_j-d}},
 \label{eq:clique-probability}
\end{align}
where $j+r_{j+1}=r_j$.

Put $T_j=n^{2-d/r_j}s^{-r_{j+1}/r_j}$, so that
$M_d=\max_{1\le j\le d-1}T_j$, and let
$m=\lfloor M_d/\omega(n)\rfloor$, where $\omega(n)\to\infty$.
Since $s\le n-1$ and $j\le d-1$, we have
$T_js/n^2=(s^j/n^d)^{1/r_j}\le1$.
Thus $ms\le n^2/\omega(n)\le n(n-1)$ for large $n$, as required.
Choosing $j$ attaining the maximum in $M_d$,
equation~\eqref{eq:clique-probability} gives
$\Prob(K_d\subseteq G_m^s)\le C'_d\omega(n)^{-r_j}=o(1)$.
The bound is uniform in $s$ and in Builder's strategy, proving
the theorem.
\end{proof}

Each choice of $j$ gives a bound valid for every Builder strategy
and every $s$; we use whichever is strongest. For $d=4$, the choices
$j=1,2,3$ give the scales $n^{4/3}s^{-5/6}$,
$n^{6/5}s^{-3/5}$ and $n^{2/3}$, respectively. These govern the
three ranges of Theorem~\ref{thm:intro-K4}: $s\le n^{4/7}$,
$n^{4/7}\le s\le n^{8/9}$, and $s\ge n^{8/9}$.

\bigskip
We conclude by mentioning two directions for further research. It would be interesting to obtain results for triangle factors and to establish asymptotically tight bounds for cliques both in $s$-semi-random and Achlioptas processes.

\section*{Declaration of AI Assistance}
Authors used ChatGPT-5.5 Pro (OpenAI) to explore the lower bound for the case $d=4$, providing their manuscript as input. The generated argument did not contain the ideas used in the present proof of Theorem \ref{thm:main-clique-lower}. The authors subsequently developed the proof for general $d$ independently. Finally, ChatGPT-6 Pro was used to condense the authors' complete and correct proof of Theorem~3 for the two cases $1\le s\le n^{4/7}$ and $n^{4/7}\le s\le n^{8/9}$, thereby improving readability.

\section*{Acknowledgements} 

The research work of Akhmejanova is supported by the Wenner-Gren Foundations, grant number UPD2024-0154.

\end{document}